\documentclass[12pt]{article}

\usepackage{fullpage,amsmath,amssymb,amsthm,float}
\usepackage[mathscr]{eucal}

\def\Enn{{\mathbb N}}

\def\fmod#1 #2{#1\ ({\rm mod}\ #2)}

\def\clo{{\rm Cl}}
\def\comp{{\rm c}}

\usepackage[dvips]{epsfig}
\usepackage{epsf}

\title{Kuratowski's Theorem for Two Closure Operators}

\author{Jeffrey Shallit \\
School of Computer Science \\
University of Waterloo \\
Waterloo, ON  N2L 3G1 \\
Canada \\
{\tt shallit@cs.uwaterloo.ca} \\
\and
Ross Willard\\
Department of Pure Mathematics\\
University of Waterloo \\
Waterloo, ON  N2L 3G1 \\
Canada \\
{\tt rdwillar@uwaterloo.ca}}

\theoremstyle{plain}
\newtheorem{theorem}{Theorem}

\newtheorem{lemma}[theorem]{Lemma}
\newtheorem{proposition}[theorem]{Proposition}

\theoremstyle{definition}
\newtheorem{definition}[theorem]{Definition}
\newtheorem{example}[theorem]{Example}

\newtheorem{question}{Question}

\theoremstyle{remark}
\newtheorem{remark}[theorem]{Remark}

\begin{document}

\maketitle

\begin{abstract}
A celebrated 1922 theorem of Kuratowski states that there are at most
$14$ distinct sets arising from applying the operations of 
complementation and closure, any number of times, in any order, to a subset of 
a topological space.  In this paper we consider the case of
complementation and {\it two\/} abstract closure operators.
In contrast to the case of a single closure operation, we show
that infinitely many distinct sets can be generated, even when the closure
operators commute.
\end{abstract}

\section{Introduction}

Let $S$ be a topological space, and let $A$ be a subset of $S$.  Let
$\clo(A)$ denote the topological closure of $A$, and $\comp(A)$ denote
$S - A$, the complement of $A$.
In 1922, Kuratowski \cite{Kuratowski:1922} observed that if we start
with an arbitrary $A$, and then apply the operations $\clo, \comp$ in
any order, any number of times, at most $14$ distinct sets are generated.
More precisely, the monoid of operations generated by
$\clo$ and $\comp$ is of cardinality $14$.  We call this monoid the
{\it Kuratowski monoid}.

However, as Hammer \cite{Hammer:1960} observed, we do not really need
all the axioms of a topological space; the same result holds in
a more abstract setting.  Let $S$ be a set, and let $k: 2^S \rightarrow 2^S$
be a map such that for all $A, B \subseteq S$, we have
\begin{enumerate}
\item $A \subseteq k(A)$; (the {\it expanding} property)

\item $A \subseteq B \implies k(A) \subseteq k(B)$; (the {\it inclusion-preserving} property)

\item $k(k(A)) = k(A)$ ({\it idempotence}).
\end{enumerate}
We call such a map a {\it closure operator}.   
We can now consider the monoid $M$ generated by $k$ and $c$ under
composition.  The identity element of this monoid is denoted by $\epsilon$.
We denote composition by concatenation so that,
for example, $kck(A) = k(c(k(A)))$.  Two elements $f, g$ of $M$ are equal if
$f(A) = g(A)$ for all $A \subseteq S$.  

Furthermore, there is a natural
partial order on elements of $M$, given by $f \leq g \iff
f(A) \subseteq g(A)$ for all $A \subseteq S$.    If $f \leq g$ and
$g \leq f$, we write $f \equiv g$.    

Hammer \cite{Hammer:1960} showed that $kckckck \equiv kck$.
(This follows immediately
from our Theorem~\ref{kura} below.)
It now follows that the monoid
generated by $\lbrace k, c \rbrace$ is
$$ \lbrace \epsilon, k, c, kc, ck, kck, ckc, kckc, ckck, kckck, ckckc, kckckc, ckckck, ckckckc \rbrace $$ 
and hence is of cardinality $14$.

In this note we consider what happens for the case of
{\it two\/} closure operators $x$ and $y$.

For readers wanting to know practically everything about the Kuratowski
theorem and its generalizations,
the admirable 
survey of Gardner and Jackson \cite{Gardner&Jackson:2008} is essential
reading.

\section{Two closure operators}

As mentioned above, it is well-known that a single closure operator
$k$ satisfies the relation $kckckck = kck$.  This result can be
easily generalized 
to {\it two} closure operators, as follows \cite{Charlier&Domaratzki&Harju&Shallit:2011}:

\begin{theorem}
Let $p, q$ be closure operators.  Then $pcqcpcq \equiv pcq$.
\label{kura}
\end{theorem}

\begin{proof}
$pcqcpcq \subseteq pcq$:  We have $L \subseteq q(L)$ by the expanding
property.  Then $cq(L) \subseteq c(L)$.  By the inclusion-preserving
property we have $pcq(L) \subseteq pc(L)$.  Since this identity holds
for all $L$, it holds in particular for $cpcq(L)$.  Substituting, we
get $pcqcpcq(L) \subseteq pccpcq(L)$.  But $pccpcq(L) = pcq(L)$ by
the idempotence of $p$.

$pcq \subseteq pcqcpcq$:  We have $L \subseteq p(L)$ by the expanding
property.  Then, replacing $L$ by $cq(L)$, 
we get $cq \subseteq pcq$.  Applying $c$ to both sides, we get
$cpcq \subseteq ccq = q$.  Applying $q$ to both sides, and using
the inclusion-preserving property and idempotence, we get
$qcpcq \subseteq qq = q$.  Applying $c$ to both sides, we get
$cq \subseteq cqcpcq$.  Finally, applying $p$ to both sides and using
the inclusion-preserving property, we get $pcq \subseteq pcqcpcq$.
\end{proof}

\begin{remark}
Theorem~\ref{kura} would also hold if $c$ were replaced by any
inclusion-reversing operation satisfying $cc \equiv \epsilon$.
\end{remark}

    Unlike the case of a single closure operator, 
our identity $pcqcpcq \equiv pcq$ does not suffice to prove that
the monoid generated by $\lbrace c, p, q \rbrace$ is finite.  Indeed,
if $p$ and $q$ have no relations between them, then there is no obvious
reason why any two distinct prefixes of $pqpqpq\cdots$ would be related.

\begin{example}
We construct a simple example where the monoid generated by
$\lbrace p, q \rbrace$ is infinite.

Consider $\Enn = \lbrace 0,1,2, \ldots \rbrace$.  For 
$A \subseteq N$, define
\begin{eqnarray*}
p(A) &=& \begin{cases}
A, & \text{ if } A = \emptyset \text{ or } \sup A = \infty; \\
A \cup \lbrace (\sup A) + 1 \rbrace, & \text{ if } \sup A \text{ is odd; } \\
A, & \text{ otherwise; } 
\end{cases} \\
q(A) &=& \begin{cases}
A, & \text{ if } A = \emptyset \text{ or } \sup A = \infty; \\
A \cup \lbrace (\sup A) + 1 \rbrace, & \text{ if } \sup A \text{ is even; } \\
A, & \text{ otherwise. }
\end{cases}
\end{eqnarray*}
Then it is easy to see that $p$ and $q$ are closure operators and
$(pq)^n (\lbrace 0 \rbrace) = \lbrace 0, 1, \ldots, 2n \rbrace$.  It follows
that the monoid generated by $\lbrace p, q \rbrace$ is infinite.
\label{jeff}
\end{example}

In order then for the monoid generated by $\{c,p,q\}$ to be finite,
we would need additional restrictions
on the closure operators
$p$ and $q$.  A natural restriction is to demand that $p$ and $q$
{\it commute}; that is, $pq \equiv qp$.  

It turns out that the case of two commuting closure operators is quite
interesting.  For example, one quickly finds additional identities,
such as (just to list a few):
\begin{eqnarray*}
pqcpcqcqcpcpq & \equiv & pqcpq \\
pqcpcpcqcqcpq & \equiv & pqcpq \\
pqcqcqcpcpcpcqcpq & \equiv & pqcpq \\
pqcqcpcpcpcqcpqcpq & \equiv &  pqcpq \\
pqcqcqcpcqcqcpcqcqcpq & \equiv & pqcpq \\
pqcpcpcqcpcpcqcpcpcpq & \equiv & pqcpq .  \\
\end{eqnarray*}

In this paper, we
will show two results:  there are infinitely many identities of this kind,
but nevertheless there are still examples where the monoid generated
is infinite.

\newcommand{\set}[2]{\{#1\,:\,\text{#2}\}}
\newcommand{\m}[1]{{\mathbf{\uppercase{#1}}}}
\newcommand{\linb}{\rule{.08in}{0in}\rule{0in}{.06in}}
\newcommand{\psf}{\mathit{psf}}
\newcommand{\barx}{\overline{x}}
\newcommand{\nospcf}{\mathit{f}}

\section{Infinitely many identities}

The goal of this section is to prove the following result.

\begin{theorem} \label{infin-eq}
Let $p,q$ be commuting closure operators.  For all $n \geq 1$ and $a_1,a_2,\ldots,a_{2n}
\in \{p,q,pq\}$, we have $pqca_1ca_2c\cdots a_{2n}cpq \equiv pqcpq$.
\end{theorem}

\begin{proof}
First note that $pq$ is also a closure operator, and for each $a \in \{p,q,pq\}$ the operator
$cac$ is an \emph{interior operator}, i.e., it is idempotent and inclusion-preserving and
satisfies the \emph{contracting} property: $A \supseteq cac(A)$.  For clarity let $1$ denote
the identity operator.  It will suffice to prove $cpqca_1ca_2c\cdots a_{2n}cpq \equiv cpqcpq$.
We have $pq\equiv pqa_2a_4\cdots a_{2n}$ and hence for any set $A$,
\begin{eqnarray*}
cpqcpq(A) &=& c(pqa_2a_4\cdots a_{2n})cpq(A)\\
&=& cpq(cc)a_2(cc)a_4\cdots (cc) a_{2n}cpq(A)\\
&=& (cpqc)1(ca_2c)1(ca_4c)1\cdots 1(ca_{2n}c)pq(A)\\
&\supseteq& (cpqc)a_1(ca_2c)a_3(ca_4c)a_5\cdots a_{2n-1}(ca_{2n}c)pq(A)
\end{eqnarray*}
on the one hand, while
\begin{eqnarray*}
cpqcpq(A) &=& cpqc(a_1a_3\cdots a_{2n-1}pq)(A)\\
&=& (cpqc)a_11a_31\cdots a_{2n-1}1pq(A)\\
&\subseteq& (cpqc)a_1(ca_2c)a_3(ca_4c)\cdots a_{2n-1}(ca_{2n}c)pq(A)
\end{eqnarray*}
on the other, proving $cpqcpq(A) = cpqca_1ca_2c\cdots a_{2n}cpq(A)$.
\end{proof}

It may be of interest to note that the above reasoning can be carried out within the
following first-order theory.

\begin{definition}
Let $T_{\rm 2com}$ be the theory with constants $1,p,q$,
binary operation $\cdot$ (written informally as juxtaposition), unary operation 
$\overline{\linb}$, binary relation $\leq$,
and axioms that state that if $(M,1,\cdot, \overline{\linb},\leq)$ is a
model of $T_{\rm 2com}$ then:
\begin{enumerate}
\item
$(M,1,\cdot)$ is a monoid.
\item
$(M,\leq)$ is a poset.
\item
$x \leq y$ and $u \leq v$ imply $xu \leq yv$ for all $x,y,u,v \in M$.
\item
$\overline{\overline{x}} = x$.
\item
$\overline{xy} = \overline{x}\cdot \overline{y}$.
\item
$x \leq y$ implies $\overline{y} \leq \overline{x}$.
\item
$\overline{1}=1$.
\item
$1 \leq p=pp$ and $1\leq q=qq$.
\item
$pq=qp$.
\end{enumerate}
\end{definition}

The intended models of $T_{\rm 2com}$ are defined as follows.  For a nonempty set $S$, let $M(S)$ denote the
set of all inclusion-preserving maps $f:2^S \rightarrow 2^S$.
Let $1$ denote the identity map $2^S\rightarrow 2^S$.  For
$f,g \in M(S)$ define
\begin{itemize}
\item
$f \leq g$ iff $f(A)\subseteq g(A)$ for all $A \subseteq S$.
\item
$f\cdot g = f\circ g$.
\item
$\overline{g}(A) = S \setminus g(S
\setminus A)$.  (i.e., $\overline{g} = c \circ g \circ c$ where $c(A) := S \setminus A$ is
the complementation operator.)
\end{itemize}
Note that $c \not\in M(S)$.  Finally, choose two commuting closure operators $p,q$ on $S$.
Then $(M(S), 1,p,q,\cdot,\overline{\linb},\leq)$ is a model of $T_{\rm 2com}$.  (Of course not all models
of $T_{\rm 2com}$ have this form.)

Given $a_0,a_1,\ldots,a_{2n+1} \in \{p,q,pq\}$ we call the word
$w=ca_0ca_1ca_2\cdots ca_{2n+1}$ in $\{p,q,c\}^\ast$ \emph{$c$-balanced} and 
associate with it
the term $t(w):=\overline{a_0}a_1\overline{a_2}a_3\cdots \overline{a_{2n}}a_{2n+1}$ in the language
of $T_{\rm 2com}$.
Observe that if $w_1,w_2$ are $c$-balanced words in $\{p,q,c\}$, then 
$T_{\rm 2com}\models t(w_1)=t(w_2)$ implies $w_1\equiv w_2$ whenever $p,q$ are
commuting closure operators $p,q$.  Hence an alternative proof of Theorem~\ref{infin-eq} can be
provided by proving

\begin{proposition}
For any $a_1,a_2,\ldots,a_{2n} \in \{p,q,pq\}$,
$T_{\rm 2com} \models \overline{pq}a_1\overline{a_2}a_3\cdots \overline{a_{2n}}pq = \overline{pq}pq$.
\end{proposition}

\begin{question}
Suppose $w_1,w_2$ are $c$-balanced words in $\{p,q,c\}^\ast$ with the property that
$w_1\equiv w_2$ whenever $p,q$ are commuting closure operators on a set $S$.
Does it follow that $T_{\rm 2com} \models t(w_1)=t(w_2)$?
\end{question}

\section{An example with infinite generated monoid}

\newcommand{\Even}{\mbox{\textsc{Even}}}
\newcommand{\Odd}{\mbox{\textsc{Odd}}}

In this section improve Example 3 by constructing a pair of commuting closure operators $p,q$
such that the monoid generated by $\{c,p,q\}$ is infinite.

For $n\geq 1$ define $w_n$ to be the word $(cpcpcqcq)^n$.  We
will construct two commuting
closure operators $p,q$ on an infinite set $S$ so that, interpreting $c$ as complementation,
there exists a subset $A \subseteq S$ with the property that $\set{w_n(A)}{$n \geq 1$}$
is infinite.

Define $\Even = \set{2n}{$n \in \mathbb Z$}$ and $\Odd = \mathbb Z \setminus \Even$.
We first define four pairs $(p_{ij},q_{ij})$ of commuting closure operators
($i,j \in \{0,1\}$) on $\mathbb Z$ as follows:  for $A \subseteq \mathbb Z$,
\begin{enumerate}
\item
$p_{00}(A)=q_{00}(A)=A$.
\item
$p_{11}(A)=q_{11}(A)= \mathbb Z$.
\item
$p_{10}(\emptyset)=\emptyset$, $p_{10}(\{2n-1\}) = p_{10}(\{2n\}) =p_{10}(\{2n-1,2n\})
= \{2n-1,2n\}$, and $p_{10}(A)=\mathbb Z$ if there does not exist $n$ with $A \subseteq
\{2n-1,2n\}$.

\noindent
$q_{10}(\emptyset)=\emptyset$, $q_{10}(\{2n\}) = q_{10}(\{2n+1\}) =q_{10}(\{2n,2n+1\})
= \{2n,2n+1\}$, and $q_{10}(A)=\mathbb Z$ if there does not exist $n$ with $A \subseteq
\{2n,2n+1\}$.
\item
$p_{01}(A) = A \cup \Odd$ and $q_{01}(A) = A \cup \Even$.
\end{enumerate}

\begin{lemma}
For all $i,j \in \{0,1\}$,
$p_{ij},q_{ij}$ are closure operators on $\mathbb Z$ and $p_{ij}q_{ij}=q_{ij}p_{ij}$.
\end{lemma}

Put $S = \mathbb Z \cup \{\top,\bot\}$.  Define $p,q:2^S \rightarrow 
2^S$ as follows: for $A \subseteq S$,
\begin{eqnarray*}
p(A) = \left\{\begin{array}{cl}
p_{00}(A), & \mbox{if $A \cap \{\top,\bot\}=\emptyset$};\\
p_{10}(A \cap \mathbb Z) \cup \{\top\} & \mbox{if $A \cap \{\top,\bot\}=\{\top\}$};\\
p_{01}(A \cap \mathbb Z) \cup \{\bot\}& \mbox{if $A \cap \{\top,\bot\}=\{\bot\}$};\\
p_{11}(A \cap \mathbb Z) \cup \{\top,\bot\}& \mbox{if $A \cap \{\top,\bot\}=\{\top,\bot\}.$}
\end{array}\right.
\end{eqnarray*}
and similarly for $q$.  Since $p(A) \cap \{\top,\bot\} = q(A) \cap \{\top,\bot\} =
 A \cap \{\top,\bot\}$ for all $A\subseteq S$, 
and since $p_{00} \leq p_{10},p_{01} \leq p_{11}$ and similarly for the
$q_{ij}$,
we can deduce from the previous lemma
that $p,q$ are commuting closure operators on $S$.
Observe that
\begin{eqnarray*}
q(\{2n,\top\}) &=& q_{10}(\{2n\}) \cup \{\top\} ~=~ \{2n,2n+1,\top\}\\
c(\{2n,2n+1,\top\}) &=& (\mathbb Z\setminus \{2n,2n+1\}) \cup \{\bot\}\\
q((\mathbb Z\setminus \{2n,2n+1\}) \cup \{\bot\}) &=& q_{01}(\mathbb Z\setminus \{2n,2n+1\}) 
\cup \{\bot\}
~=~ (\mathbb Z\setminus \{2n+1\}) \cup \{\bot\}\\
c((\mathbb Z\setminus \{2n+1\}) \cup \{\bot\}) &=& \{2n+1,\top\}.
\end{eqnarray*}
Hence $cqcq(\{2n,\top\}) = \{2n+1,\top\}$.  A similar calculation shows
$cpcp(\{2n+1,\top\}) = \{2n+2,\top\}$.  Hence if $A=\{0,\top\}$ then $w_n(A)=\{2n,\top\}$
for all $n \geq 1$, as desired.  We have shown

\begin{theorem}
The operator $cpcpcqcq$ repeatedly 
applied to $S = \mathbb Z \cup \{\top,\bot\}$
results in infinitely many distinct sets.
\end{theorem}

\section{Conclusion}

This paper leaves many questions unanswered.  We end with two questions.
Let $\Sigma$ be the set of all formal equations $w_1 = w_2$ where $w_1,w_2$ are words
in $\{c,p,q\}^\ast$ with $w_1\equiv w_2$ whenever $p,q$ are commuting closure operators.

\begin{question}
Is the set $\Sigma$ decidable?
\end{question}

\begin{question}
Can all the members of $\Sigma$ be deduced from some finitely many of them?  More precisely, is
the monoid with generators $c,p,q$ and set of relations $\Sigma$ finitely presented?
\end{question}

\end{document}